\documentclass{article}

\usepackage{graphicx}
\usepackage{color}
\usepackage{indentfirst}
\usepackage{amsmath,amsfonts,amsthm,amssymb}
\usepackage{mathrsfs}
\usepackage{amscd}
\usepackage{hyperref}

\def\co{\colon\thinspace}
\DeclareMathAlphabet{\mathsfsl}{OT1}{cmss}{m}{sl}

\newcommand{\spin}{\mathrm{Spin}^c}

\newtheorem{thm}{Theorem}[section]
\newtheorem{lem}[thm]{Lemma}

\newtheorem{prop}[thm]{Proposition}
\newtheorem{conj}[thm]{Conjecture}
\newtheorem*{thm*}{Theorem}

\theoremstyle{definition}

\newtheorem{rem}[thm]{Remark}

\newtheorem{adde}[thm]{Addendum}

\begin{document}

\def\G{{\Gamma}}
  \def\d{{\delta}}
  \def\ci{{\circ}}
  \def\e{{\epsilon}}
  \def\l{{\lambda}}
  \def\L{{\Lambda}}
  \def\m{{\mu}}
  \def\n{{\nu}}
  \def\o{{\omega}}
  \def\O{{\Omega}}
  \def\Th{{\Theta}}\def\s{{\sigma}}
  \def\v{{\varphi}}
  \def\a{{\alpha}}
  \def\b{{\beta}}
  \def\p{{\partial}}
  \def\r{{\rho}}
  \def\ra{{\rightarrow}}
  \def\lra{{\longrightarrow}}
  \def\g{{\gamma}}
  \def\D{{\Delta}}
  \def\La{{\Leftarrow}}
  \def\Ra{{\Rightarrow}}
  \def\x{{\xi}}
  \def\c{{\mathbb C}}
  \def\z{{\mathbb Z}}
  \def\2{{\mathbb Z_2}}
  \def\q{{\mathbb Q}}
  \def\t{{\tau}}
  \def\u{{\upsilon}}
  \def\th{{\theta}}
  \def\la{{\leftarrow}}
  \def\lla{{\longleftarrow}}
  \def\da{{\downarrow}}
  \def\ua{{\uparrow}}
  \def\nwa{{\nwtarrow}}
  \def\swa{{\swarrow}}
  \def\nea{{\netarrow}}
  \def\sea{{\searrow}}
  \def\hla{{\hookleftarrow}}
  \def\hra{{\hookrightarrow}}
  \def\sl{{SL(2,\mathbb C)}}
  \def\ps{{PSL(2,\mathbb C)}}
  \def\qed{{\hspace{2mm}{\small $\diamondsuit$}\goodbreak}}
  \def\pf{{\noindent{\bf Proof.\hspace{2mm}}}}
  \def\ni{{\noindent}}
  \def\sm{{{\mbox{\tiny M}}}}
   \def\ll{{{\mbox{\tiny L}}}}
    \def\lq{{{\mbox{\tiny Q}}}}
   \def\sf{{{\mbox{\tiny F}}}}
   \def\sc{{{\mbox{\tiny C}}}}
  \def\ke{{\mbox{ker}(H_1(\partial M;\2)\ra H_1(M;\2))}}
  \def\et{{\mbox{\hspace{1.5mm}}}}
 \def\sk{{{\mbox{\tiny K}}}}
\def\sxz{{{\mbox{\tiny $X_0$}}}}
\def\sxo{{{\mbox{\tiny $X_1$}}}}
\def\sxi{{{\mbox{\tiny $X_i$}}}}
\def\sxt{{{\mbox{\tiny $X_2$}}}}

\title{Dehn surgery on knots in $S^3$ producing Nil Seifert fibred spaces}

\author{{Yi Ni}\\{\normalsize Department of Mathematics, Caltech}\\
{\normalsize 1200 E California Blvd, Pasadena, CA 91125}
\\{\small\it Email\/:\quad\rm yini@caltech.edu}
\\\\
{Xingru Zhang}
\\
{\normalsize Department of Mathematics,
University at Buffalo}\\
{\small\it Email\/:\quad\rm xinzhang@buffalo.edu}}

\date{}
\maketitle

\begin{abstract}We prove that there are exactly $6$ Nil Seifert fibred spaces which can be obtained by  Dehn surgeries  on non-trefoil knots in $S^3$, with $\{60, 144, 156, 288, 300\}$ as the exact set of all such surgery slopes up to taking the mirror images of the knots.
We conjecture  that there are exactly $4$ specific
hyperbolic knots in $S^3$ which admit  Nil Seifert fibred surgery.
We also give some more general results and a more general conjecture concerning Seifert fibred surgeries on hyperbolic knots in $S^3$.
\end{abstract}

\section{Introduction}\label{sect:Intro}

For a knot  $K$  in $S^3$, we denote by $S_K^3(p/q)$ the manifold obtained by Dehn surgery along  $K$ with slope $p/q$. Here the slope $p/q$
is parameterized by the standard meridian/longitude coordinates of $K$
and we always assume $\gcd(p,q)=1$.
 In this paper we study the problem of on which knots in $S^3$
with which slopes Dehn surgeries can produce Seifert fibred spaces
admitting   the Nil geometry.
Recall that every closed connected orientable Seifert fibred space
$W$ admits one of $6$ canonical geometries:
$S^2\times \mathbb R$, $\mathbb E^3$, $\mathbb H^2\times \mathbb R$,
$S^3$, $Nil$, $\widetilde{SL}_2(\mathbb R)$.
More concretely
if $e(W)$ denotes  the Euler number of $W$ and
 $\chi({\cal B}_W)$ denotes the orbifold Euler characteristic
 of the base orbifold ${\cal B}_W$
 of $W$, then the geometry of $W$ is uniquely determined
 by the values of $e(W)$ and $\chi({\cal B}_W)$
 according to the following table
   (cf. \S 4 of \cite{Scott}):
\vspace{.3cm}
\begin{table}[h]
\begin{center}
\begin{tabular}{|c||c|c|c|} \hline
 & $\displaystyle \chi({\cal B}_W) > 0$ & $\displaystyle \chi({\cal B}_W) = 0$ &
 $\displaystyle \chi({\cal B}_W) < 0$ \\  \hline
\hline
$\displaystyle e(W) = 0$     &   $\displaystyle \mathbb S^2 \times \mathbb R$
 &  $\displaystyle \mathbb E^3$ & $\displaystyle \mathbb H^2 \times \mathbb R$ \\ \hline
$\displaystyle e(W) \ne 0$       & $\displaystyle \mathbb S^3$    &
$\displaystyle Nil$ & $\displaystyle \widetilde{SL}_2(\mathbb R)$ \\ \hline
\end{tabular}
\caption{The type of geometry of a Seifert fibred space $W$}
\end{center}
\end{table}

Suppose that $S^3_K(p/q)$ is a Seifert fibred space
with Euclidean base orbifold.
A simple homology consideration shows that the base orbifold of $S^3_K(p/q)$ must be $S^2(2,3,6)$ -- the $2$-sphere with $3$ cone points of
 orders $2,3,6$ respectively.
The orbifold fundamental group of $S^2(2,3,6)$ is the triangle group
$\triangle(2,3,6)=\langle x,y; x^2=y^3=(xy)^6=1\rangle$, whose first homology is $\z/6\z$.
Thus $p$ is divisible by $6$.
If $p=0$, then $S^3_K(0)$ must be  a torus bundle.
By \cite{G3}, $K$ is a fibred knot with genus one. So $K$ is the trefoil knot or the figure 8 knot.
But the $0$-surgery on the figure 8 knot is a manifold with the Sol geometry.
So $K$ is the trefoil knot, which means that the trefoil knot is the only knot in $S^3$  and $0$ is the only slope which can produce a Seifert fibred space with the Euclidean geometry.
Therefore  we may assume that $p\ne 0$. Hence $S^3_K(p/q)$ is a  Seifert fibred space  with the Nil geometry.
It is known that on a hyperbolic knot $K$ in $S^3$, there is at most one
 surgery  which can possibly
produce a Seifert fibred space admitting the Nil geometry and if there is one, the surgery slope is integral \cite{Boyer}.
In this paper we show

\begin{thm}\label{euclidean}Suppose $K$ is a knot in $S^3$ which is not the  (righthanded or lefthanded) trefoil knot $T(\pm 3, 2)$.
Suppose that $S^3_K(p/q)$ is a Seifert fibred space admitting the Nil geometry (where we may assume
$p, q>0$ up to changing $K$ to its mirror image).
 Then $q=1$ and $p$ is one of the numbers
 $60, 144, 156, 288, 300$.
 Moreover we have  \newline
(1)  $S^3_K(60)\cong-S^3_{T(3,2)}(60/11)$,
 \newline
 (2)  $S^3_K(144)\cong-S^3_{T(3,2)}(144/23)$
 or    $ S^3_K(144)\cong S^3_{T(3,2)}(144/25)$,
 \newline
 (3)  $S^3_K(156)\cong S^3_{T(3,2)}(156/25)$,
 \newline
 (4)   $S^3_K(288)\cong S^3_{T(3,2)}(288/49)$,
 \newline
  (5)  $S^3_K(300)\cong S^3_{T(3,2)}(300/49)$,
  \newline
  where $\cong$ stands for orientation preserving homeomorphism.
\end{thm}

Furthermore under the assumptions  of Theorem~\ref{euclidean}, we have the following additional information:

 \begin{adde}\label{Addendum} (a) The knot $K$ is either a hyperbolic knot or a cable over $T(3,2)$ as given in Proposition \ref{satellite}.

(b) If Case (1) occurs, then $K$ is a hyperbolic knot and its Alexander polynomial is  either
\newline$\triangle_K(t)=1 - t - t^{-1} + t^2 + t^{-2} - t^4 - t^{-4} + t^5 + t^{-5} - t^6 -
 t^6 + t^7 + t^{-7} - t^8 - t^{-8} + t^9 + t^{-9} - t^{13} - t^{-13} +
  t^{14} + t^{-14}
 - t^{15} - t^{-15} + t^{16} + t^{-16} - t^{22} - t^{-22} + t^{23} +
 t^{-23}$, 
 \newline or
\newline$\triangle_K(t)= 1 - t^2 - t^{-2} + t^4 + t^{-4} - t^7 - t^{-7} + t^9 + t^{-9} -
 t^{12} - t^{-12} + t^{13} + t^{-13}
 - t^{16} - t^{-16} + t^{17} + t^{-17} - t^{21} - t^{-21} + t^{22} +
 t^{-22}$.

The two Berge knots which yield the lens spaces $L(61,13)$ and $L(59,27)$
respectively realize the Nil Seifert surgery with the
prescribed two Alexander polynomials respectively.
More explicitly these two Berge knots are given in
\cite{Berge}, page 6, with $a=5$ and $b=4$ in case of Fig. 8,
and with $b=9$ and $a=2$ in case of Fig. 9, respectively.

(c) If the former  subcase of Case (2) occurs, then $K$ is a hyperbolic knot and its Alexander polynomial
is 
\newline$\triangle_K(t)=1 - t - t^{-1} + t^2 + t^{-2} - t^4 - t^{-4} + t^5 + t^{-5} - t^6 -
 t^6 + t^7 + t^{-7} - t^9 - t^{-9} + t^{10} + t^{-10} - t^{11} -
 t^{-11} + t^{12} + t^{-12} - t^{14} - t^{-14}
 + t^{15} + t^{-15} - t^{16} - t^{-16} + t^{17} + t^{-17} - t^{19} -
 t^{-19} + t^{20} + t^{-20} - t^{21} - t^{-21} + t^{22} + t^{-22} -
 t^{24} - t^{-24}
 + t^{25} + t^{-25} - t^{26} - t^{-26} + t^{27} + t^{-27} - t^{29} -
 t^{-29} + t^{30} + t^{-30} - t^{34} - t^{-34} + t^{35} + t^{-35} -
 t^{39} - t^{-39}
 + t^{40} + t^{-40} - t^{44} - t^{-44} + t^{45} + t^{-45} - t^{49} -
 t^{-49} + t^{50} + t^{-50} - t^{54} - t^{-54} + t^{55} + t^{-55}$.\newline
This case is realized on the Eudave-Mu\~noz knot
$k(-2, 1, 6, 0)$ of \cite[Propositions 5.3 (1) and  5.4 (2)]{Eudave}, which is also a Berge knot on which the $143$--surgery yields $L(143,25)$.

(d) If the latter  subcase of Case (2) occurs, then $\triangle_K(t)=\triangle_{T(29,5)}(t)\triangle_{T(3,2)}(t^5)$.
If Case (4) or (5) occurs, then
   $\triangle_K(t)=\triangle_{T(41,7)}(t)\triangle_{T(3,2)}(t^7)$
or $\triangle_K(t)=\triangle_{T(43,7)}(t)\triangle_{T(3,2)}(t^7)$  respectively.
All these cases can  be realized on certain cables
over $T(3,2)$  as given
  in Proposition~\ref{satellite}.

(e) If Case (3) occurs, then either
$\triangle_K(t)=\triangle_{T(31,5)}(t)\triangle_{T(3,2)}(t^5)$
or\newline
$\triangle_K(t)=1-t^3-t^{-3}+t^4+t^{-4}-t^5-t^{-5}+t^6+t^{-6}-t^8-t^{-8}+t^9+t^{-9}-t^{10}-t^{-10}+t^{11}+t^{-11}-t^{13}-t^{-13}+t^{14}+t^{-14}-t^{15}-t^{-15}+t^{16}+t^{-16}-t^{18}-t^{-18}+t^{19}+t^{-19}-t^{20}-t^{-20}+t^{21}+t^{-21}-t^{23}-t^{-23}+t^{24}+t^{-24}-t^{25}-t^{-25}+t^{26}+t^{-26}-t^{28}-t^{-28}+t^{29}+t^{-29}-t^{30}-t^{-30}+t^{31}+t^{-31}-t^{35}-t^{-35}+t^{36}+t^{-36}-t^{40}-t^{-40}+t^{41}+t^{-41}-t^{45}-t^{-45}+t^{46}+t^{-46}-t^{50}-t^{-50}+t^{51}+t^{-51}-t^{55}-t^{-55}+t^{56}+t^{-56}-t^{60}-t^{-60}+t^{61}+t^{-61}$.\newline
The former subcase can be  realized on the $(31,5)$-cable
over $T(3,2)$, and the latter subcase
can be realized on the Eudave-Mu\~noz knot
$k(-3, -1, 7, 0)$, which is also a Berge knot on which the $157$--surgery yields $L(157,25)$.
\end{adde}

In other words there are exactly $6$ Nil Seifert fibred spaces which can
be obtained by Dehn surgeries
on non-trefoil knots in $S^3$ and there are exactly  $5$
 slopes for all such surgeries (while on the
trefoil knot $T(3,2)$,  infinitely many Nil Seifert fibred spaces can be obtained by Dehn surgeries,
in fact by \cite{Moser}, $S_{T(3,2)}(p/q)$ is a Nil Seifert fibred space if and only if
$p=6q\pm 6$, $p\ne 0$).
It seems reasonable to  raise the following conjecture.

\begin{conj}
If a hyperbolic knot $K$ in $S^3$
admits a surgery yielding a Nil Seifert fibred space,
then $K$ is one of the four hyperbolic Berge knots given in (b) (c) (e) of Addendum \ref{Addendum}.
\end{conj}

The method of proof of Theorem \ref{euclidean}
and Addendum \ref{Addendum}
follows that  given in \cite{LiNi} and \cite{Gu},
where similar results are obtained for Dehn surgeries on knots in $S^3$
yielding spherical space forms which are not lens spaces
or prism manifolds.
The main ingredient of the method  is the use  of the correction terms
(also known as the $d$-invariants) for rational homology spheres together with their spin$^c$ structures,
 defined in \cite{OSzAbGr}.
In fact with the same method we can go a bit further
to prove the following theorems.

\begin{thm}\label{23r or 34r}For each fixed $2$-orbifold $S^2(2,3,r)$
(or $S^2(3,4,r)$), where $r>1$ is an integer satisfying
$\sqrt{6r/Q}\notin \z$
(resp. $\sqrt{12r/Q}\notin \z$)  for each $Q=1,2,...,8$,  there are only finitely many slopes with which
Dehn surgeries on hyperbolic knots
in $S^3$ can produce Seifert fibred spaces with
$S^2(2,3,r)$
(resp. $S^2(3,4,r)$) as the base orbifold.
\end{thm}

\begin{thm}\label{surgeries on torus knots}
For each fixed torus knot $T(m,n)$, with $m\geq 2$ even, $n>1$, $\gcd(m,n)=1$,
and a fixed integer $r>1$ satisfying
$\sqrt{mnr/Q}\notin \z$ for each $Q=1,2...,8$, among all
Seifert fibred spaces $$\{S^3_{T(m,n)}(\frac{mnq\pm r}{q});\;q>0, \gcd(q, r)=1 \}$$
only finitely many of them can be obtained by
Dehn surgeries on hyperbolic knots in $S^3$.
\end{thm}

The above results suggest a possible
phenomenon about Dehn surgery on hyperbolic  knots in
$S^3$ producing
Seifert fibred spaces, which we put forward in a form of
conjecture.

\begin{conj}
For every fixed $2$-orbifold $S^2(k,l, m)$, with all $k,l, m$ larger than $1$, there
are only finitely many slopes with which Dehn surgeries
on hyperbolic knots in $S^3$ can produce
Seifert fibred spaces with $S^2(k,l,m)$ as the base orbifold.
\end{conj}

In the above conjecture we may assume that $\gcd(k,l,m)=1$.

After recall
some basic properties of the correction terms in
Section~\ref{section on correction terms},
 we give and prove a more general theorem in Section~\ref{section of finiteness}.
This theorem  together with its proof will
 be applied in the proofs of Theorem \ref{euclidean}, Addendum \ref{Addendum} and Theorems \ref{23r or 34r} and \ref{surgeries on torus knots}, which is the content of Section~\ref{section 234 and 34r}.

\vspace{5pt}\noindent{\bf Acknowledgements.}\quad  The first author was
partially supported by NSF grant
numbers DMS-1103976, DMS-1252992, and an Alfred P. Sloan Research Fellowship.


\section{Correction terms in Heegaard Floer homology}
\label{section on correction terms}

To any oriented rational homology $3$-sphere $Y$ equipped with a
 Spin$^c$ structure $\mathfrak s\in\spin(Y)$, there can be
assigned  a  numerical invariant
$d(Y, \mathfrak s)\in \mathbb Q$, called the {\it correction term}
 of $(Y, \mathfrak s)$,  which is derived  in  \cite{OSzAbGr} from Heegaard Floer  homology machinery.
 The correction terms satisfy the following symmetries:
\begin{equation}\label{eq:dSymm}
d(Y,\mathfrak s)=d(Y,J\mathfrak s),\quad d(-Y,\mathfrak s)=-d(Y,\mathfrak s),
\end{equation}
where $J\co \spin(Y)\to\spin(Y)$ is the conjugation.

Suppose that $Y$ is an oriented  homology $3$-sphere, $K\subset Y$  a knot,
let $Y_K(p/q)$ be the oriented  manifold obtained by Dehn surgery on $Y$ along $K$ with slope $p/q$, where the orientation of $Y_K(p/q)$ is induced from that of $Y-K$ which in turn is induced from the given orientation of $Y$.
There is an affine isomorphism $\sigma\co\mathbb Z/p\mathbb Z\to\spin(Y_K(p/q))$. See \cite{OSzAbGr,OSzRatSurg} for more details about the isomorphism.
We shall identify $\spin(Y_K(p/q))$ with $\mathbb Z/p\mathbb Z$ via $\sigma$
but with $\s$ suppressed,  writing
$(Y_K(p/q), i)$ for $(Y_K(p/q), \s(i))$.
Note here $i$ is mod (p) defined and sometimes it can appear as an integer
larger than or equal to $p$.
The following lemma is contained in \cite{OwSt1,LiNi}.

\begin{lem}\label{lem:J}
The conjugation  $J: Spin^c(Y_K(p/q))\ra Spin^c(Y_K(p/q))$ is
given by $$J(i)=p+q-1-i, \;\;\mbox{for $0\leq i<p+q$}.$$
\end{lem}

For a positive integer $n$ and an integer $k$ we use
$[k]_n\in \z/n\z$ to denote  the congruence
class of $k$ modulo $n$.

Let $L(p,q)$ be the lens space obtained by $p/q$--surgery on the unknot in $S^3$.
The correction terms for lens spaces can be computed inductively as in \cite{OSzAbGr}:
\begin{eqnarray}\label{d invariants of L(p,q)}
d(S^3,0)&=&0,\nonumber\\
d(L(p,q),i)&=&-\frac14+\frac{(2i+1-p-q)^2}{4pq}-d(L(q,[p]_q),[i]_q),
\;\mbox{for  $0\le i<p+q$.}\label{eq:CorrRecurs}
\end{eqnarray}

For a knot $K$ in $S^3$, write its Alexander polynomial in the following standard form:
$$\triangle_K(t)=a_0+\sum_{i\geq 1} a_i(t^i+t^{-i}).$$
For $i\geq 0$, define
 $$b_i=\sum_{j=1}^{\infty}ja_{i+j}.$$
Note that the $a_i$'s can be recovered from the $b_i$'s by the following formula

\begin{equation}\label{ai from bi}
a_i=b_{i-1}-2b_i+b_{i+1},\; \mbox{for $i>0$}.
\end{equation}

By \cite{OSzRatSurg} \cite{RasThesis}, if $K\subset S^3$ is a knot
on which some Dehn surgery produces an $L$-space, then the $b_i$'s for $K$
satisfy the following properties:
\begin{equation}
b_i\geq 0, \; b_i\geq b_{i+1}\geq b_i-1, \; b_i=0\; \mbox{for $i\geq g(K)$}
\end{equation}
and  if $S^3_K(p/q)$
is an $L$-space, where $p, q>0$,
then for $0\leq i\leq p-1$,
\begin{equation}\label{d invariant}d(S^3_K(p/q), i)=d(L(p,q), i)-2b_{\min\{\lfloor\frac{i}q\rfloor,\lfloor\frac{p+q-i-1}q\rfloor\}}.
\end{equation}
This surgery formula has been  generalized in \cite{NiWu}
to one that applies to any knot in $S^3$ as follows.
Given any knot  $K$  in $S^3$, from the knot Floer chain complex,
there is a uniquely defined sequence of  integers $V_i^\sk$, $i\in\mathbb Z$, satisfying
\begin{equation}\label{eq:Vi}
V_i^\sk\geq 0, V_i^\sk\geq V^\sk_{i+1}\geq V^\sk_i-1, \quad V_i^\sk=0\; \mbox{for $i\geq g(K)$}
\end{equation}
and the following surgery formula holds

\begin{prop}\label{prop:Corr}
When $p,q>0$,
$$d(S^3_{K}(p/q),i)=d(L(p,q),i)-2V^\sk_{\min\{\lfloor\frac{i}q\rfloor,\lfloor\frac{p+q-i-1}q\rfloor\}}$$
for  $0\le i\le p-1$. \end{prop}


\section{Finitely many slopes}\label{section of finiteness}

Theorems \ref{euclidean}, \ref{23r or 34r} and \ref{surgeries on torus knots} will follow from the following more general theorem and its proof.

\begin{thm}\label{thm:BoundSlope}
Let $L$ be a given knot in $S^3$, and $r,l,Q$ be given  positive integers satisfying \begin{equation}\label{eq:NotSquare}
\sqrt{\frac{rl}Q}\notin \mathbb Z.
\end{equation}
Suppose further that $l$ is even. Then
there exist only finitely many positive integers $q$, such that $S^3_L(\frac{lq\pm r}q)$ is homeomorphic to  $S^3_K(\frac{lq\pm r}{Q})$
 for a knot $K$ in $S^3$.
\end{thm}

\begin{rem}
The condition that $l$ is even is not essential. We require this condition to simplify our argument. The condition (\ref{eq:NotSquare}) does not seem to be essential  either.
\end{rem}

We now proceed to prove Theorem \ref{thm:BoundSlope}.
Let each of  $\zeta$ and $\varepsilon$ denote an element in  $\{1,-1\}$, and let $p=lq+\zeta r$. We may assume that $p$ is positive (as long as  $q>r/l$).
Assume that
\begin{equation}
S^3_K(\frac pQ)\cong\varepsilon   S^3_L(\frac pq),
\end{equation}
where $\e\in\{\pm 1\}$ indicts the orientation and ``$\cong$'' stands for orientation preserving homeomorphism.
Then the two sets $$\{d(S^3_K(p/Q),i)|\:i\in\mathbb Z/p\mathbb Z\}, \quad\{d(\varepsilon S^3_L(p/q),i)|\:i\in\mathbb Z/p\mathbb Z\}$$
are of course  equal, but the two parametrizations for $\spin$ may not be equal: they could differ by an affine isomorphism of $\mathbb Z/p\mathbb Z$,
that is, there exists an affine isomorphism $\phi\co\mathbb Z/p\mathbb Z\to\mathbb Z/p\mathbb Z$, such that
$$d(S^3_{K}(\frac pQ),i)=d(\varepsilon S^3_L(\frac pq),\phi(i)),\;\;\mbox{for $i\in \mathbb Z/p\mathbb Z$}.$$
By Lemma~\ref{lem:J},
the fixed point set of the conjugation isomorphism
 $J\co \spin(S^3_K(p/Q))\to \spin(S^3_K(p/Q))$ is
$$\{\frac{Q-1}2,\frac{p+Q-1}2\big\}\cap\mathbb Z$$
and likewise the fixed point set of
 $J\co \spin(\varepsilon S^3_L(p/q))\to \spin(\varepsilon S^3_L(p/q))$ is
$$\{\frac{q-1}2,\frac{p+q-1}2\big\}\cap\mathbb Z.$$
As $J$  and $\phi$ commute, we must have
 $$\phi(\{\frac{Q-1}2,\frac{p+Q-1}2\big\}\cap\mathbb Z)=\{\frac{q-1}2,\frac{p+q-1}2\big\}\cap\mathbb Z.$$
It follows that the affine isomorphism $\phi:\mathbb Z/p\mathbb Z\to\mathbb Z/p\mathbb Z$ is of the form
\begin{equation}\label{eq:phi}
\phi_a(i)=[a(i-b)+\frac{(1-\alpha)p+q-1}2]_p
\end{equation}
where $b$ is an element of  $\{\frac{Q-1}2,\frac{p+Q-1}2\big\}\cap\mathbb Z$,
$\alpha=0$ or $1$,  and $a$ is an integer satisfying $0<a<p$, $\gcd(a,p)=1$.
By (\ref{eq:dSymm}) and Lemma~\ref{lem:J}, $d(\varepsilon S^3_L(p/q),\phi_a(i))=d(\varepsilon S^3_L(p/q),\phi_{p-a}(i))$.
So we may further assume that
\begin{equation}\label{eq:aCond}
0<a<\frac p2,\quad \gcd(p,a)=1.
\end{equation}

Let
\begin{equation}\label{eq:delta_a}
\delta^{\varepsilon}_a(i)=d(L(p,Q),i)-\varepsilon d(S^3_L(p/q),\phi_a(i)).
\end{equation}
By Proposition~\ref{prop:Corr}, we  have, when $q>r/l$ (so that $p>0$),
\begin{equation}\label{eq:Expected}
\delta^{\varepsilon}_a(i)=2V^\sk_{\text{min}\{\lfloor \frac{i}{Q} \rfloor,\lfloor \frac{p+Q-1-i}{Q} \rfloor \}}.
\end{equation}

Let $m\in\mathbb Z$ satisfy that
$$0\le a+\frac{(1-\alpha)\zeta r+q-1}2-mq<q,$$
then as $0<a<p/2$, we have $0\le m\le \frac l2$ when $q>2r$.

Let
$$\kappa(i)=\min\left\{\lfloor\frac{i}q\rfloor,\lfloor\frac{p+q-1-i}q\rfloor\right\}.$$

Using Proposition~\ref{prop:Corr} and (\ref{eq:delta_a}), we get
\begin{eqnarray}
\delta^{\varepsilon}_a(i)&=&d(L(p,Q),i)-\varepsilon d(S^3_L(p/q),\phi_a(i))\nonumber\\
&=&d(L(p,Q),i)-\varepsilon d(L(p,q),\phi_a(i))+2\varepsilon V^\ll_{\kappa(\phi_a(i))}.\label{eq:delta^eps}
\end{eqnarray}

\begin{lem}\label{lem:aEstimate}With the  notations and conditions
established above, there exists a constant $N=N(r,l,Q,L)$, such that
$$\left|a-\frac{mp}l\right|<N\sqrt{p}$$
for all $q>2r$.
\end{lem}
\begin{proof}
It follows from (\ref{eq:Vi}) and (\ref{eq:Expected}) that
\begin{equation}\label{eq:delta02}
\delta^{\varepsilon}_a(b+1)-\delta^{\varepsilon}_a(b)=0\text{ or }\pm2.
\end{equation}

Using (\ref{eq:delta^eps}), (\ref{d invariants of L(p,q)}) and (\ref{eq:phi}), we get
\begin{eqnarray}
&&\delta^{\varepsilon}_a(b+1)-\delta^{\varepsilon}_a(b)\nonumber\\
&=&
\frac{2b+2-p-Q}{pQ}-d(L(Q,[p]_\lq),[b+1]_\lq)+d(L(Q,[p]_\lq),[b]_\lq)+2\varepsilon( V^\ll_{\kappa(\phi_a(b+1))}- V^\ll_{\kappa(\phi_a(b))})\nonumber\\
&&-\varepsilon \big(d(L(p,q),a+\frac{(1-\alpha)p+q-1}2)- d(L(p,q),\frac{(1-\alpha)p+q-1}2)\big).\label{eq:Recurs1}
\end{eqnarray}

When $\zeta=1$, by the recursive formula (\ref{eq:CorrRecurs}), we have
(note that $a+\frac{(1-\alpha)p+q-1}2<p+q$)
\begin{eqnarray*}
&&d(L(p,q),a+\frac{(1-\alpha)p+q-1}2)- d(L(p,q),\frac{(1-\alpha)p+q-1}2)\\
&=&\frac{(2a-\alpha p)^2-(\alpha p)^2}{4pq}-d(L(q,r),a-mq+\frac{(1-\alpha) r+q-1}2)+d(L(q,r),\frac{(1-\alpha) r+q-1}2)\\
&=&\frac{a^2-a\alpha p}{pq}-\frac{(2a-2mq-\alpha r)^2-(\alpha r)^2}{4qr}\\
&&+d(L(r,[q]_r),[a-mq+\frac{(1-\alpha)r+q-1}2]_r)-d(L(r,[q]_r),
[\frac{(1-\alpha)r+q-1}2]_r)\\
&=&-\frac{l}{pr}(a-\frac{mp}l)^2+\frac{m^2}l-m\alpha\\
&&+d(L(r,[q]_r),[a-mq+\frac{(1-\alpha)r+q-1}2]_r)-d(L(r,[q]_r),
[\frac{(1-\alpha)r+q-1}2]_r).
\end{eqnarray*}
When $\zeta=-1$,
\begin{eqnarray*}
&&d(L(p,q),a+\frac{(1-\alpha)p+q-1}2)- d(L(p,q),\frac{(1-\alpha)p+q-1}2)\\
&=&\frac{(2a-\alpha p)^2-(\alpha p)^2}{4pq}-d(L(q,q-r),a-mq+\frac{-(1-\alpha) r+q-1}2)+d(L(q,q-r),\frac{-(1-\alpha) r+q-1}2)\\
&=&\frac{a^2-a\alpha p}{pq}-\frac{(2a-2mq+\alpha r-q)^2-(\alpha r-q)^2}{4q(q-r)}\\
&&+d(L(q-r,r),a-mq+\frac{-(1-\alpha)r+q-1}2)-d(L(q-r,r),\frac{-(1-\alpha)r+q-1}2)\\
&=&\frac{a^2}{pq}-\frac{\alpha a}q-\frac{(a-mq+\alpha r-q)(a-mq)}{q(q-r)}+\frac{(a-mq-(1-\alpha)r)(a-mq)}{(q-r)r}\\
&&-d(L(r,[q-r]_r),[a-mq+\frac{-(1-\alpha)r+q-1}2]_r)+
d(L(r,[q-r]_r), [\frac{-(1-\alpha)r+q-1}2]_r)\\
&=&\frac{l}{pr}(a-\frac{mp}l)^2+\frac{m^2}l-m\alpha\\
&&-d(L(r, [q]_r), [a-mq+\frac{-(1-\alpha)r+q-1}2]_r)+d(L(r,[q]_r),[\frac{-(1-\alpha)r+q-1}2]_r).
\end{eqnarray*}

Let
\begin{eqnarray*}
C_0&=&\frac{2b+2-p-Q}{pQ}-d(L(Q,[p]_\lq),[b+1]_\lq)+d(L(Q,[p]_\lq),[b]_\lq)+
2\varepsilon( V^\ll_{\kappa(\phi_a(b+1))}- V^\ll_{\kappa(\phi_a(b))})\\
&&-\varepsilon\zeta\left(
d(L(r, [q]_r), [a-mq+\frac{\zeta(1-\alpha)r+q-1}2]_r)-d(L(r,[q]_r),[\frac{\zeta(1-\alpha)r+q-1}2]_r)\right),
\end{eqnarray*}
then the right hand side of (\ref{eq:Recurs1}) becomes
$$\varepsilon\left(\zeta\frac{l}{pr}(a-\frac{mp}l)^2-\frac{m^2}l+m\alpha\right)+C_0$$

Using (\ref{eq:delta02}), we get
$$\frac{l}{pr}(a-\frac{mp}l)^2\le 2+|\frac{m^2}l-m\alpha|+|C_0|.$$
Clearly, $|C_0|$ and $m$ are bounded in terms of $r,l, Q,L$,
so the  conclusion of the lemma follows.
\end{proof}

\begin{lem}\label{lem:Ak+B+C}
Let $k$ be an integer satisfying
\begin{equation}\label{eq:kRange}
0\le k<\frac{p-(2l+1)r+l}{2Nl^2\sqrt{p}}-\frac1l.
\end{equation}
Let $$i_k=\frac{(1-\alpha)p+q-1}2+k(al-mp),\quad j_k=\frac{(1-\alpha)\zeta r+q-1}2+k(al-mp).$$
Then
$$\delta^{\varepsilon}_a(b+lk+1)-\delta^{\varepsilon}_a(b+lk)=Ak+B+C_k,$$
where
\begin{eqnarray*}
A&=&\varepsilon\zeta\cdot\frac{2(al-mp)^2}{pr}+\frac{2l}{pQ},\\
B&=&\varepsilon\left(\zeta\frac{l}{pr}(a-\frac{mp}l)^2-\frac{m^2}l+m\alpha\right),\\
C_k&=&\frac{2b+2-p-Q}{pQ}-d(L(Q,[p]_\lq),[b+lk+1]_\lq)+d(L(Q,[p]_\lq),[b+lk]_\lq)\\
&&+2\varepsilon(  V^\ll_{\kappa(\phi_a(b+lk+1))}- V^\ll_{\kappa(\phi_a(b+lk))})-\varepsilon\zeta\big(d(L(r,[q]_r),[a-mq+j_k]_r)-d(L(r,[q]_r),
[j_k]_r)\big).
\end{eqnarray*}
\end{lem}
\begin{proof}
By  (\ref{eq:kRange}), we have
\begin{equation}\label{eq:lk+1bound}
(lk+1)N\sqrt{p}<\frac{p-(2l+1)r+l}{2l}\le\frac{q-2r+1}2.
\end{equation}

It follows from (\ref{eq:aCond}), (\ref{eq:lk+1bound}) and Lemma~\ref{lem:aEstimate} that
\begin{equation}\label{eq:ikjk}
0\le i_k< i_k+a<p+q,\qquad 0\le j_k,j_k+a-mq<q.
\end{equation}
For example,
\begin{eqnarray*}
j_k+a-mq&=&j_k+a-m\frac{p-\zeta r}l\\
&=&\frac{(1-\alpha)\zeta r+q-1}2+(lk+1)(a-\frac{mp}l)+\frac{m\zeta r}l\\
&<&\frac{r+q-1}2+\frac{q-2r+1}2+\frac r2\\
&=&q.
\end{eqnarray*}
The other inequalities can be verified similarly.

Using (\ref{eq:delta^eps}), we can compute
\begin{eqnarray}
&&\delta^{\varepsilon}_a(b+lk+1)-\delta^{\varepsilon}_a(b+lk)\nonumber\\
&=&\frac{2b+2lk+2-p-Q}{pQ}-d(L(Q,[p]_\lq),[b+lk+1]_\lq)+d(L(Q,[p]_\lq),[b+lk]_\lq)\nonumber\\
&&+2\varepsilon( V^\ll_{\kappa(\phi_a(b+lk+1))}- V^\ll_{\kappa(\phi_a(b+lk))})-\varepsilon \big(d(L(p,q),i_k+a)- d(L(p,q),i_k)\big).\label{eq:Recurs2}
\end{eqnarray}

As in the proof of Lemma~\ref{lem:aEstimate}, using (\ref{eq:ikjk}) and the recursion formula  (\ref{eq:CorrRecurs}), when $\zeta=1$, we can compute
\begin{eqnarray*}
&&d(L(p,q),i_k+a)- d(L(p,q),i_k)\\
&=&\frac{(2i_k+2a+1-p-q)^2-(2i_k+1-p-q)^2}{4pq}-d(L(q,r),j_k+a-mq)+d(L(q,r),j_k)\\
&=&\frac{a(2k(al-mp)+a-\alpha p)}{pq}-\frac{(2j_k+2a-2mq+1-q-r)^2-(2j_k+1-q-r)^2}{4qr}\\
&&+d(L(r,[q]_r),[j_k+a-mq]_r)-d(L(r,[q]_r),[j_k]_r)\\
&=&-\frac{2(al-mp)^2}{pr}k-\frac{l}{pr}(a-\frac{mp}l)^2+\frac{m^2}l-
m\alpha+d(L(r,[q]_r),[j_k+a-mq]_r)-d(L(r,[q]_r),[j_k]_r).
\end{eqnarray*}
Similarly, when $\zeta=-1$, we get
\begin{eqnarray*}
&&d(L(p,q),i_k+a)- d(L(p,q),i_k)\\
&=&\frac{2(al-mp)^2}{pr}k+\frac{l}{pr}(a-\frac{mp}l)^2+\frac{m^2}l-
m\alpha-d(L(r,[q]_r),[j_k+a-mq]_r)+d(L(r,[q]_r),[j_k]_r).
\end{eqnarray*}
So the right hand side of (\ref{eq:Recurs2}) is $Ak+B+C_k$.
\end{proof}

We can now finish the proof of Theorem~\ref{thm:BoundSlope}.
If $S^3_{K}(p/Q)\cong \varepsilon S^3_L(p/q)$, then (\ref{eq:Expected}) holds, so
\begin{equation}\label{eq:delta02lk}
\delta^{\varepsilon}_a(b+lk+1)-\delta^{\varepsilon}_a(b+lk)=0\text{ or }\pm2
\end{equation}
for all $k$ satisfying (\ref{eq:kRange}).

Let $A,B,C_k$ be as in Lemma~\ref{lem:Ak+B+C}. By (\ref{eq:NotSquare}),
 $A\ne0$. So $Ak+B+C$ is equal to 0 or $\pm 2$ for at most three values of $k$ for any given $C$.
From the expression of $C_k$, it is evident that
 there exists a constant integer $M=M(L)$, such that given $p,q,a,\varepsilon,\zeta$, as $k$ varies, $C_k$ can take at most $MQr$ values.  Thus  $Ak+B+C_k$ can be 0 or $\pm2$,
  i.e. (\ref{eq:delta02lk})  holds,  for
 at most $3MQr$ values of $k$.
But if $p\ge 4l^2N^2(3lMQr+2)^2$, then
 each of $k$ in $\{0,1,2,\dots,3MQr\}$ satisfies (\ref{eq:kRange})
 and thus (\ref{eq:delta02lk}) holds for each of these $3MQr+1$ values of  $k$.
This contradiction shows that $p$  is bounded above by $4l^2N^2(3lMQr+2)^2$.


\section{Seifert surgeries}\label{section 234 and 34r}

In this section we prove Theorem \ref{euclidean}, Addendum  \ref{Addendum}
and Theorems \ref{23r or 34r} and \ref{surgeries on torus knots}.

\begin{lem}\label{23 or 34}If $W$ is an oriented   Seifert fibred space whose base orbifold is $S^2(2,3,r)$ (or $S^2(3,4,r)$), $r>1$, then $W$ is
homeomorphic to some surgery on the torus  knot $T(3,2)$ (resp.
$T(4,3)$), i.e.  $$W\cong \varepsilon S^3_{T(3,2)}(\frac{6q+\zeta r}{q})
\;\;\; (\mbox{resp.}\;\; W\cong \varepsilon S^3_{T(4,3)}(\frac{12q+\zeta r}{q}))$$
for some $\varepsilon,\zeta\in \{1, -1\}$ and some positive integer $q$.
\end{lem}

\begin{proof} The proof is a quick generalization of that of \cite[Lemma 3.1]{LiNi}.
The Seifert space $W$ has three singular fibres of orders
$2,3, r$ (resp. $3,4,r$)  respectively. The exterior of the singular fiber of order $r$ in $W$ is  homeomorphic (not necessarily orientation preserving)  to the exterior of the torus  knot $T(3,2)$ (resp. $T(4,3)$) in $S^3$
because there is only one Seifert fibred space (up to homeomorphism) with base orbifold $D^2(2,3)$ (resp.  $D^2(3,4)$). Now  on $T(3,2)$ (resp. $T(4,3)$),
a surgery gives Seifert fibred space with  base orbifold $S^2(2,3,r)$ (resp. $S^2(3,4, r)$) if and only if
the slope is $\frac{6q+\zeta r}{q}$ (resp. $\frac{12q+\zeta r}{q}$),  $\gcd(q,r)=1$. We may assume $q>0$ up to change the sign of $\zeta$.
\end{proof}

The following proposition classifies satellite knots in $S^3$ which admit
Nil Seifert surgeries.

\begin{prop}\label{satellite}
Suppose $K$ is a satellite knot and $S^3_K(p/q)$ is a Nil Seifert fibred space with $p/q>0$.
Then $K$ is a cable over $T(3,2)$. More precisely, there are four cases for the cable type and the slope:
$$
\begin{tabular}{|c|c|}
\hline
\text{cable type} & p/q\\
\hline
(29,5) &144/1\\
\hline
(31,5) &156/1\\
\hline
(41,7) &288/1\\
\hline
(43,7) &300/1\\
\hline
\end{tabular}
$$
\end{prop}

\begin{proof}
Let $C$ be a companion knot of $K$ such that $C$ is itself not a satellite knot.
Let $V$ be a solid torus neighborhood of $C$ in $S^3$ such that
 $K$ is contained in the interior of $V$ but is not contained in a $3$-ball in $V$
 and is not isotopic to the core circle of $V$.
Let $N$ be a regular neighborhood of $K$ in $V$,
 $M_K=S^3-int(N)$,  $M_C=S^3-int(V)$, and let $V_K(p/q)$ be the $p/q$-surgery of $V$ along $K$.
Then $S^3_K(p/q)=M_K(p/q)=M_C\cup V_K(p/q)$.
Since $S^3_K(p/q)$ does not contain incompressible tori,
$\p V$ must be compressible in $S^3_K(p/q)$
and in fact compressible in $V_K(p/q)$. 
By \cite{GSurg}, it follows that either $V_K(p/q)$ has a connected summand $W$ with $0<|H_1(W)|<\infty$, or $V_K(p/q)$ is a solid torus.
In the former case, by \cite{S} $V_K(p/q)$ contains a lens space as a connected summand, which
contradicts the fact that $S^3_K(p/q)=M_K(p/q)=M_C\cup V_K(p/q)$ is a Nil Seifert fibred space.
Hence  $V_K(p/q)$ is a solid torus.
Now by \cite{GSurg}, $K$ is a $0$ or $1$-bridge braid in $V$ with winding number $w>1$.
By  \cite[Lemma 3.3]{Go} the meridian slope of the solid torus
$V_K(p/q)$ is $p/w^2q$ and thus $M_K(p/q)=M_C(p/w^2q)$.
So $C$ is a torus knot by \cite{Boyer} and then $C$ must be the trefoil knot $T(3,2)$ by \cite{Moser}.

If $K$ is a $(s,t)$-cable in $V$ (where we may assume $t>1$ is the winding number of $K$ in $V$),
then by \cite[Lemma 7.2]{Go}, $p=stq+\e_1$, $\e_1\in\{\pm 1\}$. So
 $M_K(p/q)=M_C((stq+\e_1)/(t^2q))$.
By \cite{Moser} we should have $stq+\e_1=6t^2q+\e_2 6$, $\e_2\in\{\pm1\}$.
So we have  $stq-6t^2q=\e_1 5$ or $-\e_1 7$, which implies
 $q=1$ and $t=5$ or $7$.

If $t=5$, then $s=30+\e_1$ and $p=5(30+\e_1)+\e_1$.
That is, either $K$ is the $(29,5)$-cable over $T(3,2)$, $q=1$ and $p=144$
or  $K$ is the $(31,5)$-cable over $T(3,2)$, $q=1$  and $p=156$.
Likewise if $t=7$, $K$ is the $(41,7)$-cable over $T(3,2)$, $q=1$  and $p=288$
or
$K$ is the $(43,7)$-cable over $T(3,2)$, $q=1$  and $p=300$.

Now suppose that $K$ is a $1$-bridge braid in $V$.
By \cite[Lemma 3.2]{G1bridge}, $q=1$ and $p=\t w+d$ where $w$ is the winding number of $K$ in $V$,
and $\t$ and $d$ are integers satisfying $0<\t<w-1$ and $0<d<w$.
Hence  $M_K(p/q)=M_C(\t w+d/w^2)$ and by \cite{Moser}
$\t w+d=6w^2\pm 6$. But  $6w^2\pm 6-\t w-d\geq 6w^2-6-(w-1)w-w =5w^2-6>0$.
We get a contradiction, which means $K$ cannot be a $1$-bridge braid in $V$.
\end{proof}

\begin{proof}[Proof of Theorem \ref{euclidean} and Addendum \ref{Addendum}]
Let $K$ be any non-trefoil knot in $S^3$ such that $S^3_K(p/q)$ is a Nil
Seifert space. Up to changing $K$ to its mirror image, we may assume that
$p,q>0$.
If $K$ is a torus knot, then by \cite{Moser}, no surgery on $K$
can produce a Nil Seifert fibred space.
So we may assume that $K$ is not a torus knot.
By \cite{Boyer} and Proposition \ref{satellite},
$q=1$.
We are now going to give a concrete upper  bound for $p$.
As noted in Section~\ref{sect:Intro}, the base orbifold of
 $S^3_K(p)$ is $S^2(2,3,6)$. Thus by Lemma \ref{23 or 34}, $S^3_K(p)\cong\varepsilon S^3_{T(3,2)}(p/q)$ with
$p=6q+\zeta 6$, for some $\varepsilon, \zeta\in \{1, -1\}$, and $q>0$.
As $p\ne 0$, $p/q>1=g(T(3,2))$ which implies that
  $S^3_K(p)\cong\varepsilon S^3_{T(3,2)}(p/q)$ is an L-space by \cite[Corollary~1.4]{OSzRatSurg}.
Therefore we may use surgery formula (\ref{d invariant}) instead of
Proposition \ref{prop:Corr}.
Now we  apply the proof of Theorem \ref{thm:BoundSlope}
 (and the notations established there)  to our current case
with $L=T(3,2)$, $Q=1$, $l=r=6$.
Then $m\in\{0,1,2,3\}$, $b\in \{0, p/2\}$, $V^\ll_i=b_i^{T(3,2)}$ (which is  $1$
if $i=0$ and $0$ if $i>0$), $V^\sk_i=b_i^\sk$ and
\begin{eqnarray*}
C_0&=&\frac{2b+2-p-1}{p}+
2\varepsilon( b^{T(3,2)}_{\kappa(\phi_a(b+1))}- b^{T(3,2)}_{\kappa(\phi_a(b))})\\
&&-\varepsilon\zeta\left(
d(L(6, [q]_6), [a-mq+\frac{\zeta(1-\alpha)6+q-1}2]_6)-d(L(6,[q]_6),[\frac{\zeta(1-\alpha)6+q-1}2]_6)\right),
\end{eqnarray*}
Using formula (\ref{d invariants of L(p,q)}) one can compute
\begin{equation}\label{eq:d of L(6,q)}
d(L(6, q), i)=\left\{\begin{array}{ll}
(\frac{5}{4}, \frac{5}{12}, \frac{-1}{12},
\frac{-1}{4}, \frac{-1}{12},\frac{5}{12}), & q=1, i=0, 1,...,5,\\&\\
(\frac{-5}{12}, \frac{1}{12}, \frac{1}{4},
\frac{1}{12}, \frac{-5}{12},\frac{-5}{4}),& q=5, i=0, 1,...,5.\end{array}\right.
\end{equation}
Thus $|C_0|\leq 1+2+\frac32<5$. Since $|\frac{m^2}6-m\alpha|\le2$ for $m=0,1,2,3$ and $\alpha=0,1$, we may take $N=3$.
Similarly recall the $A, B, $ and  $C_k$ in Lemma~\ref{lem:Ak+B+C}, and in our current case, $C_k$
 becomes
\begin{eqnarray*}
C_k&=&\frac{2b+2-p-1}{p}+
2\varepsilon( b^{T(3,2)}_{\kappa(\phi_a(b+lk+1))}- b^{T(3,2)}_{\kappa(\phi_a(b+lk))})\\
&&-\varepsilon\zeta\left(
d(L(6, [q]_6), [a-mq+j_k]_6)-d(L(6,[q]_6),[j_k]_6)\right),
\end{eqnarray*}
which can take at most $18$ values as $k$ varies.
Thus  the bound for $p$ is
$4\cdot 6^2\cdot3^2(3\cdot6\cdot18+2)^2$ when $A\ne 0$.

Now we just need to show that in our current case,
 $A$ is never zero.
 Suppose otherwise that $A=0$.
Then $\varepsilon\zeta=-1$, $(a-\frac{mp}{6})^2=1$, so $a-mq=\zeta m\pm 1$, and by Lemma~\ref{lem:Ak+B+C}
$$\begin{array}{l}
\delta^{\varepsilon}_a(b+lk+1)-\delta^{\varepsilon}_a(b+lk)
=B+C_k\\
=-\varepsilon\frac{m^2}{6}+\varepsilon m\a+(\mbox{$0$ or -1})+2\varepsilon( b^{T(3,2)}_{\kappa(\phi_a(b+lk+1))}- b^{T(3,2)}_{\kappa(\phi_a(b+lk))})
\\\quad +d(L(6, [q]_6), [\zeta m\pm 1+j_k]_6)-d(L(6,[q]_6),[j_k]_6).
\end{array}$$
Thus \begin{equation}\label{lr6case}-\varepsilon\frac{m^2}{6}+
d(L(6, [q]_6), \zeta m\pm 1+[3\zeta(1-\alpha)+\frac{q-1}2]_6)-d(L(6,[q]_6),[3\zeta(1-\alpha)+\frac{q-1}2]_6)\end{equation}
is integer valued.
Using (\ref{eq:d of L(6,q)}), we see
that for each of $m=0,1,2,3$,\quad $q\equiv1,5\pmod6$, $\a\in\{0,1\}$ and $\zeta\in\{1,-1\}$,
the expression given in (\ref{lr6case})
is never integer valued.
This contradiction proves the assertion that $A\ne 0$.

Now for the bounded region of integral slopes for  $p$,  one can use computer calculation to locate those possible integral slopes
and identify the corresponding Nil Seifert fibred spaces given in Theorem \ref{euclidean}, applying
 (\ref{d invariant}) (\ref{d invariants of L(p,q)}), which yields Theorem \ref{euclidean}.
  One can also recover the possible Alexander polynomials for the candidate knots
using formula (\ref{ai from bi}). The rest of Addendum \ref{Addendum} follows
from \cite{Moser}, Proposition~\ref{satellite}, \cite{Eudave}, 
and direct verification using SnapPy.
\end{proof}

\begin{proof}[Proof of Theorem~\ref{23r or 34r}]
Let $K$ be a hyperbolic knot in $S^3$
such that $S^3_K(p/Q)$ is a Seifert fibred space whose base orbifold
is $S^2(2,3,r)$ (or $S^2(3,4,r)$).
By changing $K$ to its mirror image, we may assume that both
$p$ and $Q$ are positive integers.
By \cite{LM}  we   have $Q\leq 8$. So we just need to show that $p$ is bounded above
(independent of hyperbolic $K$).

By Lemma~\ref{23 or 34},
$$S^3_K(p/Q)\cong  \varepsilon S^3_{T(3,2)}(\frac{6q+\zeta r}{q})\;\;
(\mbox{resp.}\;S^3_K(p/Q)\cong \varepsilon S^3_{T(4,3)}(\frac{12q+\zeta r}{q}))$$
for some $\varepsilon,\zeta\in \{1, -1\}$ and some positive integer $q$.
Now applying Theorem~\ref{thm:BoundSlope} with $l=6$ and $L=T(3,2)$ (resp.  $l=12$
and $L=T(4,3)$),
our desired conclusion is true when (\ref{eq:NotSquare}) holds, i.e.
$$\sqrt{\frac{6r}{Q}}\notin \z\;\; (\mbox{resp.}\;
\sqrt{\frac{12r}{Q}}\notin \z)$$
for each $Q=1,...,8$.
\end{proof}

\begin{proof}[Proof of Theorem~\ref{surgeries on torus knots}]
Let $K$ be a hyperbolic knot in $S^3$
such that $S^3_K(p/Q)\cong \varepsilon S^3_{T(m,n)}(\frac{mnq+\zeta r}{q})$.
Again $Q\leq 8$ and  Theorem~\ref{thm:BoundSlope} applies
with $l=mn$ and $L=T(m,n)$.
\end{proof}

\end{document}